\documentclass[11pt]{amsart}

\usepackage[mathscr]{eucal}
\usepackage{amsmath,amssymb,amsfonts,amsthm,enumerate}
\usepackage{hyperref}

\newtheorem{theorem}{Theorem}[section]
\newtheorem{lemma}[theorem]{Lemma}
\newtheorem{definition}[theorem]{Definition}

\newtheorem{conjecture}[theorem]{Conjecture}

\newcommand{\Z}{\mathbb Z}

\DeclareMathOperator{\ord}{ord}

\newcommand{\la}{\langle}
\newcommand{\ra}{\rangle}
\newcommand{\be}{\begin{equation}}
\newcommand{\ee}{\end{equation}}
\newcommand{\und}{\;\mbox{ and }\;}
\newcommand{\nn}{\nonumber}
\newcommand{\ber}{\begin{eqnarray}}
\newcommand{\eer}{\end{eqnarray}}
\newcommand{\Sum}[2]{\underset{#1}{\overset{#2}{\sum}}}

\newcommand{\vp}{\mathsf v}

\DeclareSymbolFont{goo}{OMS}{cmsy}{b}{n}
\DeclareMathSymbol{\gooT}{\mathalpha}{goo}{"1}
\newcommand{\bdot}{\mathbin{\gooT}}
\begin{document}
\title[Structure of a sequence with prescribed zero-sum subsequences]{Structure of a sequence with prescribed zero-sum subsequences: Rank Two $p$-groups}
\author{John J. Ebert}
\address{Department of Mathematical Sciences\\ University of Memphis\\ Memphis, TN 38152\\
USA}
\email{john.ebert01@gmail.com}

\author{David J. Grynkiewicz}
\address{Department of Mathematical Sciences\\ University of Memphis\\ Memphis, TN 38152\\
USA}
\email{diambri@hotmail.com}

\begin{abstract}
Let $G=(\mathbb Z/n\mathbb Z) \oplus (\mathbb Z/n\mathbb Z)$. Let $\textup{s}_{\leq k}(G)$ be the smallest integer $\ell$ such that every sequence of $\ell$ terms from $G$, with repetition allowed, has a nonempty zero-sum subsequence with length at most $k$. It is known that $\textup{s}_{\leq 2n-1-k}(G)=2n-1+k$ for $k\in [0,n-1]$,  with the structure of extremal sequences showing this bound tight  determined  when $k\in \{0,1,n-1\}$, and for various  special cases when $k\in [2,n-2]$. For the remaining values $k\in [2,n-2]$, the characterization of extremal sequences of length $2n-2+k$ avoiding a nonempty zero-sum of length at most $2n-1-k$ remained open in general, with it conjectured that they must all have the form $e_1^{[n-1]} \boldsymbol{\cdot} e_2^{[n-1]} \boldsymbol{\cdot} (e_1 +e_2)^{[k]}$ for some basis $(e_1,e_2)$ for $G$. Here $x^{[n]}$ denotes a sequence consisting of the term $x$ repeated $n$ times. In this paper, we establish this conjecture for all $k\in [2,n-2]$ when $n$ is prime, which in view of other recent work, implies the conjectured structure for all rank two abelian groups.
\end{abstract}

\maketitle

\section{Introduction}

Let $C_n$ denote a cyclic group of order $n$. Let $G$ be a finite abelian group written additively. Then $G=C_{n_1}\oplus C_{n_2}\oplus \ldots \oplus C_{n_r}$ with $1<n_1 \mid n_2 \mid \ldots \mid n_r$, where  $\textup r(G)=r$ is the rank of $G$, and $\exp(G)=n_r$ is the exponent  of $G$. Following standardized notation \cite{Alfredbook} \cite{Ger-Ruzsa-book} \cite{Gbook} detailed in Section \ref{sec2}, let
$$S = g_1 \boldsymbol{\cdot} \ldots \boldsymbol{\cdot} g_{\ell}$$
be a (finite and unordered) sequence of terms $g_i \in G$, written as a multiplicative string with  repetition of terms allowed. Such a sequence is called  \emph{zero-sum} if the sum of its terms equals zero, $\sum_{i=1}^{\ell}{g_i}=0$.

The Davenport Constant of $G$ is the minimal integer $\textup D(G)$ such that any sequence of terms from $G$ with length $|S|\geq \textup D(G)$ must have a nonempty zero-sum subsequence. It is one of the most well studied combinatorial invariants in Additive Number Theory, both of interest from a purely combinatorial perspective as well due to its relevance to the study of Factorization in structures from Commutative Algebra  \cite{Alfredbook} \cite{Ger-Ruzsa-book}. Despite this, its exact value is known only for very limited groups, including $p$-groups and groups of rank at most $2$. There, it is known that  $\textup D(G)=1+\Sum{i=1}{r}(n_i-1)$ \cite{olson-pgroup} \cite{Boas-K-rk2-Dav} \cite{olson-rk2} \cite{Alfredbook}. In particular,
\begin{align*}\textup D(C_n)=n,\quad \textup D(C_n\oplus C_n)=2n-1,\quad\und \textup D(C_p\oplus C_p\oplus C_p)=3p-2,\end{align*} for any $n\geq 1$ and any $p\geq 2$ prime, which we will use implicitly throughout the paper.

The standard proof of $\textup D(C_n\oplus C_n)=2n-1$ \cite{Boas-K-rk2-Dav} \cite{olson-rk2} \cite{Alfredbook} relies upon an inductive strategy, reducing the general case to when $n=p$ is prime, and making use of the axillary invariant $\eta(G)$, defined as the
 minimal integer such that any sequence of terms from $G$ with length $|S|\geq \eta(G)$ must have a nonempty zero-sum subsequence of length at most $\exp(G)$. Later, Delorme, Ordaz and Quiroz introduced \cite{S_<k-invariant-origin-delorme-ordaz} the invariant $\textup s_{\leq k}(G)$ as a common generalization, defined as the minimal integer such that any sequence of terms from $G$ with length $|S|\geq \textup s_{\leq k}(G)$ must have a nonempty zero-sum subsequence of length at most $k$. Indeed, when $k\geq \textup D(G)$, then $\textup s_{\leq k}(G)=\textup D(G)$, and when $k=\exp(G)$, then $\textup s_{\leq k}(G)=\eta(G)$.
 The relations between $\textup{s}_{\leq k}(G)$ and Coding Theory were explored by Cohen and Zemor in \cite{Cohen}. Other related works that deal with $\textup{s}_{\leq k}(G)$ can be found in \cite{freeze} \cite{Roy-s<k} \cite{Gao-liu}. The authors in \cite{kevin-S_<k-invariant} determined $\textup{s}_{\leq k}(G)$ for all finite abelian groups of rank two. Note, since  $\textup s_{\leq k}(G)=\infty $ when $k<\exp(G)$, while $\textup s_{\leq k}(G)=\textup s_{\leq \textup D(G)}(G)$ for all $k\geq \textup D(G)$, that  $\textup s_{\leq \textup D(G)-k}(G)$ is primarily of  interest for $k\in [0,\textup D(G)-\exp(G)]=[0,m-1]$, meaning there is little need to consider values of $k$ outside this range.

\begin{theorem}[\cite{kevin-S_<k-invariant}, Theorem 2]\label{s-theorem}
Let $G = C_m \oplus C_n$, where $m$ and $n$ are integers with $1\leq m \mid n$, and let $k\in [0,m-1]$. Then
$$\textup{s}_{\leq \textup{D}(G)-k}(G)=\textup s_{\leq n+m-1-k}(G)=\textup D(G) + k = m+ n-1+k.$$
\end{theorem}

In particular, for $G=C_n \oplus C_n$, we know that
$$\textup{s}_{\leq \textup{D}(G)}(G)=\textup D(G) = 2n-1$$
and
$$\textup{s}_{\leq \textup{exp}(G)}(G)=\eta (G) = 3n-2.$$
It is then natural to ask which extremal sequences  with terms from $G$ show these bounds are tight, i.e., can those sequences $S$ with length $|S|=\textup D(G)-1+k=2n-2+k$ having no nonempty zero-sum subsequence of length at most $\textup D(G)-k=2n-1-k$ be characterized? The cases $k\in\{0,1,n-1\}$ were eventually resolved, with precise structure following due to the combined efforts from numerous papers \cite{propBGaoGer-multof2} \cite{PropB} \cite{Reiher-propB}  \cite{PropBfix} \cite{Schmid-propB}
 (See Conjecture \ref{conjecture} and Theorem \ref{probB}). The resulting characterization has proved useful in various applications, e.g.,  \cite{PropB-app8} \cite{PropB-app1} \cite{PropB-app5}   \cite{PropB-app9} \cite{PropB-app10} \cite{Prop-B-app2} \cite{PropB-app7}  \cite{PropB-app3}    \cite{PropB-app6} \cite{PropB-app12-InverseMatomakiPeng} \cite{PropB-app11} \cite{PropB-app4}. In \cite{S_<k-invariant2/3p}, the problem of characterizing the extremal sequences for the invariant $\textup s_{\leq \textup D(G)-k}(C_n\oplus C_n)$ was proposed (for $n$ prime), with the conjecture stated in \cite{S_<k-invariant2/3p} naturally extended to composite values of $n$ in \cite{inverse-mult}. The conjectured structure, including the known cases for $k\in\{0,1,n-1\}$, can be summarized as follows. Here $x^{[m]}=x\bdot\ldots\bdot x$ denotes the sequence consisting of the element $x\in G$ repeated $m$ times.

\begin{conjecture}[\cite{inverse-mult}, Conjecture 1.1]\label{conjecture} Let $n\geq 2$, let $G=C_n \oplus C_n$, let $k\in [0,n-1]$, and let $S$ be a sequence of terms from $G$ with length  $|S|=\textup D(G)+k-1=2n-2+k$ having no nonempty zero-sum subsequence of length at most  $\textup D(G)-k=2n-1-k$. Then there exists a basis $(e_1,e_2)$ for $G$ such that the following hold.
\begin{itemize}

\item[1.] If $k=0$, then $S\boldsymbol{\cdot}g$ satisfies the description given in Item 2, where $g=-\sigma(S)$.

\item[2.] If $k=1$, then $$S=e_1^{[n-1]}\boldsymbol{\cdot}{\prod}^{\bullet}_{i\in[1,n]}{(x_ie_1+e_2)},$$
for some $x_1,...,x_n \in [0,n-1]$ with $x_1+...+x_n\equiv 1 \mod{n}$.

\item[3.] If $k\in [2,n-2]$, then $$S=e_1^{[n-1]} \boldsymbol{\cdot} e_2^{[n-1]} \boldsymbol{\cdot} (e_1 + e_2)^{[k]}.$$

\item[4.] If $k=n-1$, then $$e_1^{[n-1]} \boldsymbol{\cdot} e_2^{[n-1]} \boldsymbol{\cdot} (xe_1 + e_2)^{[k]}.$$
for some $x\in[1,n-1]$ with $\gcd{(x,n)}=1$.
\end{itemize}
\end{conjecture}

\medskip

As already noted, Conjecture  \ref{conjecture} is known for $k\in \{0,1,n-1\}$, leaving the range $k\in[2,n-2]$ open. In this range, Conjecture \ref{conjecture} is known in various specialized cases, including when $k\leq \frac{2n+1}{3}$ with $n$ a prime power \cite{S_<k-invariant2/3p} \cite{inverse-mult}, as well as for several very  specialized cases derived in \cite{inverse-mult}. In \cite{inverse-mult}, it was shown how the Conjecture \ref{conjecture} holding when $n=p$ is prime would imply the general case. Specifically, the following was shown.

\begin{theorem}[\cite{inverse-mult}, Theorem 1.2]\label{Mult-Prop} Let $n,m\geq 2$ and let $k\in [0,mn-1]$ with   $k=k_mn+k_n$, where  $k_m\in [0,m-1]$ and $k_n\in [0,n-1]$. Suppose Conjecture \ref{conjecture} holds for $k_n$ in $C_{n}\oplus C_{n}$ and also  for $k_m$ in $C_{m}\oplus C_{m}$. Then Conjecture \ref{conjecture} holds for $k$ in $C_{mn}\oplus C_{mn}$.
\end{theorem}

In another recent paper \cite{inverse-mult2}, a more complicated description of all extremal sequences for a general rank two abelian group $G=C_m\oplus C_n$ was given and also shown to follow from Conjecture \ref{conjecture}. Thus the complete characterization of all extremal sequences for the invariant $\textup s_{\leq \textup D(G)-k}(C_m\oplus C_n)$ is reduced to the case
$\textup s_{\leq \textup D(G)-k}(C_p\oplus C_p)$ with $p$ prime, where it remained open for  $k\geq \frac{2p+2}{3}$. The goal of this paper is to resolve this case, establishing Item 3 in Conjecture \ref{conjecture} for all $k\in [2,n-2]$ when $n=p$ is prime, which as discussed, thereby implies Conjecture \ref{conjecture} holds without restriction, and gives the full characterization of all extremal sequences for a general rank two group. Specifically, we will show the following. As our proof does not rely on the main result from \cite{S_<k-invariant2/3p} and works equally well for all values of $k\in[2,p-2]$, this also gives a new proof of the cases $k\leq \frac{2p+1}{3}$ versus that from \cite{S_<k-invariant2/3p}, though we will use arguments and lemmas from \cite{S_<k-invariant2/3p}.

\begin{theorem}\label{main-result}
Let $G=C_p\oplus C_p$ with $p$ a prime, let $k\in [2,p-2]$ be an integer,
 and let $S$ be a sequence of terms from $G$ with $|S|=\textup D(G)+k-1=2p-2+k$ having no nonempty zero-sum subsequence of length at most $\textup D(G)-k= 2p-1-k$. Then there is a basis $(e_1,e_2)$ for $G$ such that
$$S=e_1^{[p-1]}\boldsymbol{\cdot} e_2^{[p-1]}\boldsymbol{\cdot} (e_1+e_2)^{[k]}.$$
\end{theorem}

The proof of Theorem \ref{main-result} makes use of the characterization of extremal sequences for the Davenport Constant $\textup D(C_p\oplus C_p)$, some combinatorial arguments, and the arguments from \emph{two} separate proofs of Theorem \ref{s-theorem} (when $m=n=p$ is prime): the original given in \cite{kevin-S_<k-invariant}, as well as a new one derived here and accomplished  by lifting to the group $C_p\oplus C_p \oplus C_p$. The latter is a variant on a strategy used for studying the Erd\H{o}-Ginzburg-Ziv Constant $\textup s(G)$ (see e.g. \cite[Proposition 5.8.1]{Alfredbook}), defined as the minimal integer such that any sequence of terms from $G$ with length $|S|\geq \textup s(G)$ must have a nonempty zero-sum subsequence of length  exactly $\exp(G)$.
 We do not explicitly detail the argument separately, simply remarking that the proof of Lemma \ref{height-of-S} easily modifies (when applied to an arbitrary  sequence of length $|S|=2p-1-k$ rather than a specialized one of length $|S|=2p-2+k$) to show $\textup s_{\leq 2p-1-k}(C_p\oplus C_p)=2p-1+k$.

\section{Preliminaries}\label{sec2}

We will briefly present key concepts and notation  used throughout this paper. Let $\mathbb{N}$ denote the set of positive integers and $\mathbb{N}_0=\mathbb{N} \cup \left\lbrace 0 \right\rbrace$. For $x,y\in \mathbb{R}$, we use $[x,y]=\left\lbrace z\in \mathbb{Z}: x\leq z \leq y \right\rbrace$ for the discrete interval between $x$ and $y$. We use $C_n$ to denote a cyclic group of order $n\geq 1$.

\medskip Following standardized notation for combinatorial sequences (\cite{Ger-Ruzsa-book} \cite{Alfredbook} \cite{Gbook}),
for an abelian group $G$, we let  $\mathcal{F}(G)$ be the free abelian monoid with basis $G$, whose elements consist of finite strings of terms from $G$, with the order of terms in the string disregarded.  The elements $S\in \mathcal F(G)$ are called (finite and unordered) sequences $S$ of terms from $G$, which have the form
$$S=g_1 \boldsymbol{\cdot} g_2 \boldsymbol{\cdot} ... \boldsymbol{\cdot} g_{\ell}  = {\prod}^\bullet_{i\in [1,\ell]}{g_i}\in \mathcal{F}(S),$$
with the $g_i\in G$ the terms of the sequence $S$.  For $k\geq 0$ and $g\in G$, we let $g^{[k]}=\underbrace{g\boldsymbol{\cdot} ... \boldsymbol{\cdot} g}_k$ be the sequence with the term $g$ repeating $k$ times, with $g^{[0]}$ the empty sequence consisting of no terms. Letting $$\vp_g(S)=\{i\in [1,\ell]:\; g_i=g\}\geq 0$$ denote the multiplicity of the term $g$ in $S$, we can then write $S$ as
$$S={\prod}^\bullet_{g\in G}{g ^ {[\vp_g(S)]}}\in \mathcal{F}(S).$$ If $\vp_g(S) \geq 1$, then we say that $S$ contains $g$. We call $T$ a subsequence of $S$ if $\vp_g(T) \leq \vp_g(S)$ for all $g\in G$. In such case, let $T^{[-1]}\boldsymbol{\cdot} S=S\boldsymbol{\cdot} T^{[-1]}$ denote the subsequence of $S$ obtained by removing the terms of $T$, that is,
$$T^{[-1]}\boldsymbol{\cdot} S={\prod}^{\bullet}_{g\in G}{g ^ {[\vp_g(S)-\vp_g(T)}]}\in \mathcal{F}(S).$$

\medskip

If $T\in \mathcal{F}(G)$ and $k\geq 1$, we let $T^{[k]}=\underbrace{T\boldsymbol{\cdot} ... \boldsymbol{\cdot} T}_k$ be the sequence consisting of  $T$ repeating $k$ times. If $T^{[k]}$ is a subsequence of $S$, then $T^{[-k]}\boldsymbol{\cdot} S=S\boldsymbol{\cdot} T^{[-k]}=(T^{[k]})^{[-1]}\boldsymbol{\cdot} S$. We use  the following notation:

\begin{itemize}
\item $|S|=\ell = \sum_{g\in G}\vp_g(S)\in \mathbb{N}_0$ is the length of $S$,

\item $\textup h(S) = \max \left\lbrace \vp_g(S):\; g\in G \right\rbrace$ is the maximum multiplicity of $S$,

\item $\sigma(S) = \sum_{i=1}^{\ell}{g_i} = \sum_{g\in G}{\vp_g(S) g}\in G$ is the sum of terms in $S$,

\item $\Sigma(S) = \left\lbrace \sum_{i\in I}g_i : I\subseteq [1,\ell] \mbox{ with } 1\leq |I| \leq \ell \right\rbrace$ is the set of all subsums of $S$,

\item $\Sigma_k(S) = \left\lbrace \sum_{i\in I}g_i : I\subseteq [1,\ell] \mbox{ with } |I|=k \right\rbrace$ is  the set of all length $k$ subsums of $S$,

\item $\Sigma_{\leq k}(S) =  \bigcup_{i\in [1,k]}{\Sigma_i(S)}$.

\end{itemize}

 A sequence $S$ is called
\begin{itemize}
\item \textit{zero-sum free} if $0\not \in \Sigma(S)$,

\item \textit{a zero-sum sequence} if $\sigma(S) = 0$,

\item \textit{a minimal zero-sum sequence} if $S$ is a nonempty zero-sum sequence that does not contain any proper, nonempty zero-sum subsequence.
\end{itemize}

\medskip

If $G$ and $H$ are abelian groups. Then any map $\phi : G \rightarrow H$ can be extended to a map from $\mathcal{F}(G)$ to $\mathcal{F}(H)$ by setting
$$\phi(S)= {\prod}^{\bullet}_{i\in[1,\ell]}{\phi(g_i)}.$$
 We will need the following results and definitions.

\begin{definition} Let $G$ be an abelian group, let $S=g_1 \boldsymbol{\cdot} ... \boldsymbol{\cdot} g_{\ell}\in \mathcal F(G)$ be a sequence of terms from $G$, where $\ell=|S|$, and let $k\geq 0$. Then
$$N^k(S) := |\lbrace I\subseteq [1,|S|] :\; \sum_{i\in I}{g_i} = 0 \ \textup{and} \ |I|=k \rbrace|$$ denotes the number of zero-sum subsequence of $S$ having length $k$.
\end{definition}

\begin{lemma}[\cite{Alfredbook}, Proposition 5.5.8]\label{5.5.8} Let $p$ be a prime, let $G$ be a finite abelian $p$-group, and let $S\in \mathcal F(G)$ be a sequence of terms from $G$. If $|S|\geq \textup D(G)$, then $\Sum{i=0}{|S|}(-1)^iN^i(S)\equiv 0\mod p$.
\end{lemma}

\begin{lemma}[\cite{Bounding_lemma}, Lemma 2.7]\label{bounding_lemma} Let $G$ be an abelian group and let $S\in\mathcal{F}(G)$ be a zero-sum free sequence. Then
$$\left|\Sigma{(S)}\right|\geq |S| + |\textup{supp}(S)|-1.$$
\end{lemma}

\begin{theorem}[\cite{PropB, Reiher-propB}]\label{probB} Let $G=C_n \oplus C_n$ with $n\geq 2$ and let $S\in \mathcal{F}(G)$ be a minimal zero-sum sequence with length $\textup{D}(G)=2n-1$. Then $S$ has the following form:
$$S=e_1^{[n-1]}\boldsymbol{\cdot} {\prod}^{\bullet}_{i\in[1,n]}{(x_ie_1+e_2)}$$
with $x_i\in [0,n-1]$ and $\sum_{i=1}^{n}{x_i} \equiv 1 \mod{n}$, for some basis $(e_1,e_2)$ for $G$.
\end{theorem}

\begin{lemma}[\cite{S_<k-invariant2/3p}, Lemma 15]\label{car1} Let $G=C_n \oplus C_n$, let $k\in [2,n-2]$, and let
 $$S=e_1^{[n-1]}\boldsymbol{\cdot} {\prod}^{\bullet}_{i\in[1,n+k-1]}{(x_ie_1+e_2)}\in \mathcal F(G),$$
where $x_i\in [1,n]$ for $i\in [1,n+k-1]$ and $\sum_{i=1}^n{x_i}\equiv 1 \mod{n}$. If $0\not\in \Sigma_{\leq 2n-1-k}{(S)}$, then there exists a basis $(e_1,f_2)$ for $G$, where $f_2 = xe_1 +e_2$ for some $x\in[1,n]$, such that
$$S=e_1^{[n-1]}\boldsymbol{\cdot} f_2^{[n-1]}\boldsymbol{\cdot} (e_1+f_2)^{[k]}.$$
\end{lemma}

\section{Proof of main result}
To start determining the structure of $S\in \mathcal{F}(C_p^2)$ where $|S|=2p-2+k$ and $0\not \in \Sigma_{\leq \textup{D}(C_p^2)-k}(S)$, we will first show that $S$ has a zero-sum subsequence of length $\textup{D}(C_p^2)$. To accomplish this, we will need the following two lemmas, which extend arguments used in  \cite[Lemma 14]{S_<k-invariant2/3p},  themselves based on the original proof of Theorem \ref{s-theorem} given in \cite{kevin-S_<k-invariant}.

\begin{lemma}\label{con-system} Let $p$ be a prime and $k\in [1,p-1]$. Consider the family of $k$ linear congruencies in the variables $x_1,\ldots,x_k$:
\be\label{congrency}
\binom{2p-2+k}{t} + \binom{2k-2}{t}x_1 + \binom{2k-3}{t}x_2 + \ldots + \binom{k-1}{t}x_k \equiv 0 \mod{p},
\ee
where $t\in [0,k-1]$. Then the unique solution to the above system is   $x_s \equiv (-1)^{k-s+1}\binom{k}{k-s+1} \mod{p}$ for $s\in [1,k]$.
\end{lemma}
\begin{proof}
Let $X=(1,x_1,...,x_k)^T$ and
$$A:= \begin{pmatrix}
\binom{2p-2+k}{0} & \binom{2k-2}{0} & \binom{2k-3}{0} & \ldots & \binom{k-1}{0}\\
\binom{2p-2+k}{1} & \binom{2k-2}{1} & \binom{2k-3}{1} & \ldots & \binom{k-1}{1}\\
\ldots & \ldots & \ldots & \ldots & \ldots\\
\binom{2p-2+k}{k-1} & \binom{2k-2}{k-1} & \binom{2k-3}{k-1} & \ldots & \binom{k-1}{k-1}
\end{pmatrix}.$$
From (\ref{congrency}), we have
$$AX\equiv 0 \mod{p}.$$
Since $\binom{n}{0}=1$, for any $n$, we have
$$A = A_{1,0}= \begin{pmatrix}
\binom{2p-3+k}{0} & \binom{2k-3}{0} & \binom{2k-4}{0} & \ldots & \binom{k-2}{0}\\
\binom{2p-2+k}{1} & \binom{2k-2}{1} & \binom{2k-3}{1} & \ldots & \binom{k-1}{1}\\
\ldots & \ldots & \ldots & \ldots & \ldots\\
\binom{2p-2+k}{k-1} & \binom{2k-2}{k-1} & \binom{2k-3}{k-1} & \ldots & \binom{k-1}{k-1}
\end{pmatrix}.$$
By multiplying the first row of $A_{1,0}$ by $-1$, adding it to the second row of $A_{1,0}$ and using the property $\binom{n}{i}-\binom{n-1}{i-1}=\binom{n-1}{i}$, we obtain
$$A_{1,1}= \begin{pmatrix}
\binom{2p-3+k}{0} & \binom{2k-3}{0} & \binom{2k-4}{0} & \ldots & \binom{k-2}{0}\\
\binom{2p-3+k}{1} & \binom{2k-3}{1} & \binom{2k-4}{1} & \ldots & \binom{k-2}{1}\\
\ldots & \ldots & \ldots & \ldots & \ldots\\
\binom{2p-2+k}{k-1} & \binom{2k-2}{k-1} & \binom{2k-3}{k-1} & \ldots & \binom{k-1}{k-1}
\end{pmatrix}.$$
We can repeat this process $k-1$ times. That is, for $0 \leq i \leq k-2$, multiply row $i+1$ of $A_{1,i}$ by $-1$ and add the result to row $i+2$ to construct $A_{1,i+1}$. Then
$$A_{1,k-1}= \begin{pmatrix}
\binom{2p-3+k}{0} & \binom{2k-3}{0} & \binom{2k-4}{0} & \ldots & \binom{k-2}{0}\\
\binom{2p-3+k}{1} & \binom{2k-3}{1} & \binom{2k-4}{1} & \ldots & \binom{k-2}{1}\\
\ldots & \ldots & \ldots & \ldots & \ldots\\
\binom{2p-3+k}{k-1} & \binom{2k-3}{k-1} & \binom{2k-4}{k-1} & \ldots & \binom{k-2}{k-1}
\end{pmatrix}.$$
Repeating the above technique of row operations $\ell\leq k-1$ times, we obtain
$$A_{\ell,k-1}= \begin{pmatrix}
\binom{2p-2+k-\ell}{0} & \binom{2k-2-\ell}{0} & \binom{2k-3-\ell}{0} & \ldots & \binom{k-1-\ell}{0}\\
\binom{2p-2+k-\ell}{1} & \binom{2k-2-\ell}{1} & \binom{2k-3-\ell}{1} & \ldots & \binom{k-1-\ell}{1}\\
\ldots & \ldots & \ldots & \ldots & \ldots\\
\binom{2p-2+k-\ell}{k-1} & \binom{2k-2-\ell}{k-1} & \binom{2k-3-\ell}{k-1} & \ldots & \binom{k-1-\ell}{k-1}
\end{pmatrix}.$$
Ultimately, for $\ell=k-1$, we obtain
$$A_{k-1,k-1}= \begin{pmatrix}
\binom{2p-1}{0} & \binom{k-1}{0} & \binom{k-2}{0} & \ldots & \binom{0}{0}\\
\binom{2p-1}{1} & \binom{k-1}{1} & \binom{k-2}{1} & \ldots & \binom{0}{1}\\
\ldots & \ldots & \ldots & \ldots & \ldots\\
\binom{2p-1}{k-1} & \binom{k-1}{k-1} & \binom{k-2}{k-1} & \ldots & \binom{0}{k-1}
\end{pmatrix},$$ which is simply equal to $A$ when $k=1$.
Since $AX\equiv 0 \mod{p}$ and $\binom{n}{h}=0$ when $0\leq n < h$, it follows that $AX\equiv A_{k-1,k-1}X\equiv 0 \mod{p}$. That is, for $s\in [1,k],$
$$\binom{2p-1}{k-s} + \binom{k-1}{k-s}x_1 + \binom{k-2}{k-s}x_2 + \ldots + \binom{k-s}{k-s}x_s\equiv 0 \mod{p}.$$
We will now proceed by induction on $s\in [1,k]$. By Lucas's Theorem, $\binom{2p-1}{h}\equiv \binom{p-1}{h}\equiv (-1)^h \mod{p}$ for $0\leq h\leq p-1$. When $s=1$, we have $(-1)^{k-1}+x_1\equiv \binom{2p-1}{k-1} + \binom{k-1}{k-1}x_1 \equiv 0 \mod{p}$, which implies that $x_1 \equiv (-1)^k \equiv (-1)^k \binom{k}{k} \mod{p}$. We will now assume $s\geq 2$ and that $x_h \equiv (-1)^{k-h+1}\binom{k}{k-h+1} \mod{p}$ for all $h\in [1,s-1]$. Since $\binom{2p-1}{k-s} + \binom{k-1}{k-s}x_1 + \binom{k-2}{k-s}x_2 + \ldots + \binom{k-s}{k-s}x_s\equiv 0 \mod{p}$ and $\binom{2p-1}{k-s+1} + \binom{k-1}{k-s+1}x_1 + \binom{k-2}{k-s+1}x_2 + \ldots + \binom{k-s+1}{k-s+1}x_{s-1}\equiv 0 \mod{p}$,  it follows that
\begin{align}
\nn x_s & \equiv - \binom{2p-1}{k-s+1} - \binom{2p-1}{k-s} - \sum_{h=1}^{s-1}{\left( \binom{k-h}{k-s+1} + \binom{k-h}{k-s}\right)x_h} \\\label{bird4}
& =- \binom{2p}{k-s+1}- \sum_{h=1}^{s-1}{\binom{k-h+1}{k-s+1}x_h} \\
& \label{bird1} \equiv - \sum_{h=1}^{s-1}{(-1)^{k-h+1}\binom{k-h+1}{k-s+1} \binom{k}{k-h+1}} \\\label{bird2}
& = (-1)^{k+1} \binom{k}{k-s+1}  \sum_{h=1}^{s-1}{(-1)^{h-1}\binom{s-1}{s-h}} \\\label{bird3}
& = (-1)^{k-s+1} \binom{k}{k-s+1} \mod{p},
\end{align} where \eqref{bird4} follows in view of the binomial identity $\binom{n}{i}=\binom{n-1}{i}+\binom{n-1}{i-1}$,
where \eqref{bird1} follows in view of $1\leq s\leq k\leq p-1$ and the induction hyopthesis, where  \eqref{bird2} follows in view of the binomial identity $\binom{b}{a}\binom{c}{b}=\binom{c}{a}\binom{c-a}{b-a}$ for $0\leq a\leq b$, and where \eqref{bird3} follows by evaluating the polynomial identity $(x-1)^{s-1}=\sum_{h=1}^{s-1}x^{s-h}\binom{s-1}{s-h}(-1)^{h-1}+(-1)^{s-1}$ at $x=1$, which yields the desired value for $x_s$.
\end{proof}

\begin{lemma}\label{mzs-exists} Let $G=C_p \oplus C_p$ with $p$ prime, let $k\in [1,p-1]$ be an integer, and let $S\in \mathcal{F}(G)$ be a sequence of terms from $G$ with $|S|=\textup{D}(G)+k-1=2p-2+k$ and $0 \not\in \Sigma_{\leq \textup{D}(G)-k}(S)=\Sigma_{\leq 2p-1-k}(S)$. Then the following hold.
\begin{itemize}
\item[(a)]  For all $i\in [1,\textup{D}(G)-k]\cup[\textup{D}(G)+1,2\textup{D}(G)-2k+1]=[1,2p-1-k]\cup [2p, 4p-2k-1]$, we have $N^{i}(S)=0$.
\item[(b)] $N^{\textup{D}(G)}(S)\equiv k \mod{p}$. In particular, $S$ contains at least $k$  zero-sum subsequences of length $\textup{D}(G)=2p-1$, and any such zero-sum is minimal.
\item[(c)] If $k\geq 2$, then $\sigma(S)\neq 0$.
\end{itemize}
\end{lemma}

\begin{proof}
Recall that $$\textup{D}(G)=2p-1.$$

 \noindent (a): By hypothesis, $N^{i}(S)=0$ for all $i\in [1,\textup{D}(G)-k]$. If $i\in [\textup{D}(G)+1,2\textup{D}(G)-2k+1]$ and $N^{i}(S)\neq 0$, then $S$ has a zero-sum subsequence of length $i$, say $T$. Since $i > \textup{D}(G)$, then $T$ has a nonempty zero-sum subsequence of length at most $\textup{D}(G)$, say $R$. Then $R$ and $R^{[-1]}\bdot T$ are both nonempty, proper zero-sum subsequences of $S$, and one of them has length at most $\textup{D}(G)-k$, which is contrary to hypothesis.

\medskip

\noindent  (b): Let $T$ be a subsequence of $S$ of length $|T|=|S|-t\geq 2p-1$, where $t\in[0,k-1]$.

Suppose $k\leq\dfrac{2p+1}{3}$. Then $|S|=2p-2+k\leq 4p-2k-1=\textup D(G)-2k+1$. By Lemma \ref{5.5.8} and (a), we have
$$1+\sum_{i=2p-k}^{2p-1}{(-1)^iN^{i}(T)}\equiv 0 \mod{p}.$$
From this, we have
$$\sum_{T\mid S, |T|=|S|-t}{\left(1+\sum_{i=2p-k}^{2p-1}{(-1)^iN^{i}(T)}\right)}\equiv 0 \mod{p},  \quad\mbox{ for every $t\in [0,k-1]$}.$$
By counting the number of times each zero-sum subsequence of $S$ occurs in the above sum, we obtain
\begin{equation}\label{con_1}
\binom{|S|}{|T|} + \sum_{i=2p-k}^{2p-1}{(-1)^{i} \binom{|S|-i}{|T|-i}N^i(S)} = \binom{|S|}{t} + \sum_{i=2p-k}^{2p-1}{(-1)^{i} \binom{|S|-i}{t}N^i(S)}\equiv 0 \mod{p},
\end{equation}
for every $t\in [0,k-1]$. Let us next derive a similar congruence when $k\geq \frac{2p+2}{3}$.

Suppose $k\geq\dfrac{2p+2}{3}$. Then  $|S|=2p-2+k>4p-2k-1$,  so by Lemma \ref{5.5.8} and (a),
$$1+\sum_{i=2p-k}^{2p-1}{(-1)^iN^{i}(T)}+\sum_{i=4p-2k}^{2p-2+k}{(-1)^iN^{i}(T)}\equiv 0 \mod{p}.$$
Through the same process we used when $k\leq\dfrac{2p+1}{3}$, we obtain
$$\binom{|S|}{t} + \sum_{i=2p-k}^{2p-1}{(-1)^{i} \binom{|S|-i}{t} N^i(S)} + \sum_{i=4p-2k}^{2p-2+k}{(-1)^{i} \binom{|S|-i}{t} N^i(S)} \equiv 0 \mod{p}.$$
We have $t\leq k-1\leq p-1$, so by Lucas's Theorem, we find that $\binom{|S|-i}{t}\equiv \binom{p+|S|-i}{t} \mod{p}$ for  $i\in [4p-2k,2p-2+k]$. As a result, we obtain
$$\binom{|S|}{t} + \sum_{i=2p-k}^{2p-1}{(-1)^{i} \binom{|S|-i}{t} N^i(S)} + \sum_{i=4p-2k}^{2p+k-2}{(-1)^{i} \binom{p+|S|-i}{t} N^i(S)} \equiv 0 \mod{p}.$$
By re-indexing the third summation, we obtain
$$\binom{|S|}{t} + \sum_{i=2p-k}^{2p-1}{(-1)^{i} \binom{|S|-i}{t} N^i(S)} + \sum_{i=3p-2k}^{p+k-2}{(-1)^{i+p} \binom{|S|-i}{t} N^{i+p}(S)} \equiv 0 \mod{p}.$$
Since $k<p$, then $2p-k < 3p-2k$ and $p+k-2<2p-1$,  so we obtain
 \begin{equation}\label{con_2}
\begin{split}
\binom{|S|}{t} & + \sum_{i=2p-k}^{3p-2k-1}{(-1)^{i} \binom{|S|-i}{t} N^i(S)} \\
 & + \sum_{i=3p-2k}^{p+k-2}{ \binom{|S|-i}{t}  \left( (-1)^{i} N^i(S) + (-1)^{i+p} N^{i+p}(S)\right)} \\
 & +\sum_{i=p+k-1}^{2p-1}{(-1)^{i}\binom{|S|-i}{t} N^i(S)} \equiv 0 \mod{p},
\end{split}
\end{equation}
for every $t\in [0,k-1]$.

In view of \eqref{con_1} and \eqref{con_2},
we can apply Lemma \ref{con-system}  to yield $(-1)^{2p-1} N^{2p-1} (S)\equiv -\binom{k}{1}\mod p$. Since $2p-1$ is odd, then $N^{2p-1} (S)\equiv k\mod p$. Since $k\not\equiv 0 \mod{p}$ and $N^{2p-1} (S) \geq 0$, then $N^{2p-1} (S) \geq k$. Lastly, if a zero-sum subsequence of $S$ of length $2p-1$ was not minimal, then $S$ would contain a subsequence of length at most $p\leq 2p-1-k$, which is contrary to hypothesis.

\medskip

 \noindent (c): Assume by contradiction that part (c) is false, that is, $\sigma(S)=0$. Since $k\not\equiv 0 \mod p$, then $N^{2p-1} (S) \geq 1$ by (b), so $S$ has a zero-sum subsequence of length $2p-1$, which we call $T$. Since $S$ is a zero-sum sequence, then $T^{[-1]}\bdot S$ will be a zero-sum subsequence of $S$ of length $|S|-|T|=k-1\in [1,p-2]\subseteq [1,2p-1-k]$ (for $k\geq 2$), which is contrary to hypothesis.
\end{proof}

\begin{lemma}\label{p-1_repeating} Let $G=C_p \oplus C_p$ with $p$ prime, let $k\in [1,p-2]$ be an integer, and let $S\in \mathcal{F}(G)$ be a sequence of terms from $G$ with $|S|=\textup{D}(G)+k-1=2p-2+k$ and $0 \not\in \Sigma_{\leq \textup{D}(G)-k}(S)=\Sigma_{\leq 2p-1-k}(S)$. If $(e_1,e_2)$ is a basis for $G$ such that $S=e_1^{[p-1]}\boldsymbol{\cdot} e_2^{[p-1]}\boldsymbol{\cdot} T$, then $S=e_1^{[p-1]}\boldsymbol{\cdot} e_2^{[p-1]}\boldsymbol{\cdot} (e_1 +e_2)^{[k]}$.
\end{lemma}
\begin{proof}  If $k=1$, then $|S|=2p-1=\textup D(G)$ with $0\notin \Sigma_{\leq 2p-2}(S)$ ensures that $S$ is a minimal zero-sum sequence of length $2p-1$, forcing $T=e_1+e_2$, as desired. Therefore we can assume $k\geq 2$.

Let $\overline{n}\in [1,p]$ be the least positive integer congruent to $n$ modulo $p$. Let $\phi :G\rightarrow \Z/p\Z$ be defined by $xe_1 + ye_2 \mapsto \overline{x+y-1}\mod p$. Let $T'={\prod}^{\bullet}_{i\in[1,|T'|]}{(x_i e_1 +y_i e_2)}$, where $x_i, y_i \in [1,p]$, be an arbitrary nonempty subsequence of $T$. Then $$\sigma(\phi(T'))\equiv \sum_{i\in[1,|T'|]}{\overline{x_i+y_i-1}} \equiv \overline{\sum_{i\in[1,|T'|]}{x_i}} + \overline{\sum_{i\in[1,|T'|]}{y_i}} - |T'| \mod{p}.$$
Since $S':=e_1^{[p-\overline{\sum_{i\in[1,|T'|]}{x_i}}]} \boldsymbol{\cdot} e_2^{[p-\overline{\sum_{i\in[1,|T'|]}{y_i}}]} \boldsymbol{\cdot} T'$ is a nonempty zero-sum subsequence of $S$, then by Lemma \ref{mzs-exists} parts (a) and (c), we have
$$|S'|=2p-\overline{\sum_{i\in[1,|T'|]}{x_i}} - \overline{\sum_{i\in[1,|T'|]}{y_i}} + |T'| \in [2p-k,2p-1]\cup [4p-2k, 2p-3+k].$$
From this, we find that
$$\sigma(\phi (T'))\equiv \overline{\sum_{i\in[1,|T'|]}{x_i}} + \overline{\sum_{i\in[1,|T'|]}{y_i}} - |T'| \in [1,k]\cup [3-k,2k-2p] \equiv [1,k]\cup [p+3-k,2k-p] \mod{p}.$$
Since $[p+3-k,2k-p]\subseteq[4,k-1]$, then $\sigma(\phi (T')) \in  [1,k] \mod{p}$, and as $T'$ was an arbitrary nonempty subsequence of $T$, this shows that $$\Sigma(\phi (T))\subseteq [1,k]\mod p.$$
Since $k\leq p-2$, then $-1,0 \not\in \Sigma(\phi (T))$. Thus we can apply  Lemma \ref{bounding_lemma} to  obtain $|\textup{supp}(\phi(T))|=1$, say $\phi(T)=g^{[k]}$ with $g\neq 0$.  As a result $\Sigma(\phi (T))=\{g,2g,\ldots,kg\}\subseteq [1,k]\mod  p$ is an arithmetic progression with difference $g$ and length $k\in [2,p-2]$, and thus also equal to the arithmetic progression $[1,k]$ with difference $1$ which contains it. Since an arithmetic progression with difference $g$ and length from $[2,\ord(g)-2]$ has its difference unique up to sign (as is easily verified), it follows that $g=\pm 1$, and as  $-1\notin \Sigma(\phi(T))$, we are left to conclude that $g=1$, meaning $\phi(T)=1^{[k]}$.

\medskip

So for any term of $T$, say $\alpha e_1 +\beta e_2$ where $\alpha , \beta \in [1,p]$, we have $\alpha + \beta -1 \equiv 1 \mod p$. Due to the bounds on $\alpha$ and $\beta$, it follows that $\alpha + \beta  = 2$ or $\alpha + \beta =p+2$.  If $\alpha + \beta =p+2$, then  $e_1^{[p-\alpha]} \boldsymbol{\cdot} e_2^{[p-\beta]} \boldsymbol{\cdot} (\alpha e_1 + \beta e_2)$ is a zero-sum subsequence of $S$ of length $2p-\alpha - \beta +1 = p -1\leq 2p-1-k$, contrary to hypothesis. Therefore $\alpha + \beta  = 2$, which forces $\alpha = \beta =1$ and $T=(e_1 + e_2)^{[k]}$.
\end{proof}

\begin{lemma}\label{mzs_3d} Let $G=C_p \oplus C_p \oplus C_p$ with $p$ prime and let $S\in \mathcal{F}(G)$ be a minimal zero-sum sequence of length $\textup D(G)=3p-2$. If there is an $e_1 \in G$ such that $\vp_{e_1}(S) \geq p-1$, then there exists $e_2, e_3 \in G$ such that $(e_1,e_2,e_3)$ is a basis of $G$ and $S$ has the following form:
$$S=e_1^{[p-1]}\boldsymbol{\cdot} {\prod}^{\bullet}_{i\in[1,p-1]}{(\alpha_ie_1+e_2)}\boldsymbol{\cdot} {\prod}^{\bullet}_{i\in[1,p]}{(\beta_ie_1+\gamma_ie_2+e_3)},$$
with $\alpha_i, \beta_i, \gamma_i \in [0,p-1]$ and $\sum_{i=1}^{p-1}{\alpha_i}  + \sum_{i=1}^{p}{\beta_i} \equiv \sum_{i=1}^{p}{\gamma_i} \equiv 1 \mod{p}$.
\end{lemma}

\begin{proof}
Since $S$ is a minimal zero-sum of length $3p-2>p$, we must have $\vp_{e_1}(S)\leq p-1$, whence $\vp_{e_1}(S)=p-1$.
Since $e_1 \neq 0$, there exists an $ H \leq G$ such that $G = \left< e_1 \right> \oplus H$,  so $S$ will have the form
\begin{equation}\label{first-form-3d}
S=e_1^{[p-1]}\boldsymbol{\cdot} {\prod}^{\bullet}_{i\in[1,2p-1]}{(x_ie_1+h_i)}
\end{equation}
where $h_i \in H$, $x_i\in [0,p-1]$ and, clearly, $\sum_{i=1}^{2p-1}{x_i} \equiv 1 \mod{p}$. Consider the sequence  $S' = {\prod}^{\bullet}_{i\in[1,2p-1]}{h_i}$. Since $S$ is zero-sum, it follows that $S'$ is zero-sum. Moreover, if $S'$ has a proper, nonempty zero-sum $T'$, then the corresponding subsequence of  ${\prod}^{\bullet}_{i\in[1,2p-1]}{(x_ie_1+h_i)}$ will be a proper, nonempty subsequence whose sum lies in $\la e_1\ra$, which can be made into a proper, nonempty zero-sum subsequence of $S$ by concatenating an appropriate number of terms from $e_1^{[p-1]}$. Since this would contradict that $S$ is a minimal zero-sum, we conclude that $S'$ is a minimal zero-sum of length $2p-1$ with terms from $H\cong C_p\oplus C_p$. Then by Theorem \ref{probB}, it follows that $S'$ has the form $$S'=e_2^{[p-1]}\boldsymbol{\cdot} {\prod}^{\bullet}_{i\in[1,p]}{(\gamma_ie_2+e_3)}$$
with $\gamma_i \in [0,p-1]$ and $\sum_{i=1}^{p}{\gamma_i} \equiv 1 \mod{p}$, for some basis $(e_2,e_3)$ of $H$. By re-indexing $S'$, we have that $h_i =e_2$ for $i\in [1,p-1]$ and $h_i = \gamma_{i-p+1}e_2 +e_3$ for $i\in[p,2p-1]$. By setting $\alpha_i =x_i$ for $i\in [1,p-1]$ and $\beta_i=x_{i+p-1}$ for $i\in[1,p]$, we can rewrite \eqref{first-form-3d}, and $S$ will have the form
\begin{equation}\label{final-form-3d}
S=e_1^{[p-1]}\boldsymbol{\cdot} {\prod}^{\bullet}_{i\in[1,p-1]}{(\alpha_ie_1+e_2)}\boldsymbol{\cdot} {\prod}^{\bullet}_{i\in[1,p]}{(\beta_ie_1+\gamma_ie_2+e_3)}.
\end{equation}
Since $(e_1,e_2,e_3)$ is a basis of $G$ due to $(e_2,e_3)$ being a basis of $H$, $\sum_{i=1}^{n}{\gamma_i} \equiv 1 \mod{p}$, and $\sum_{i=1}^{p-1}{\alpha_i}  + \sum_{i=1}^{p}{\beta_i} = \sum_{i=1}^{2p-1}{x_i} \equiv 1 \mod{p}$, then (\ref{final-form-3d}) has the desired properties.
\end{proof}

\begin{lemma}\label{height-of-S} Let $G=C_p \oplus C_p$ with $p$ prime, let $k\in [2,p-2]$ be an integer, and let $S\in \mathcal{F}(G)$ be a sequence of terms from $G$ with $|S|=\textup D(G)+k-1=2p-2+k$ and $0 \not\in \Sigma_{\leq \textup{D}(G)-k}(S)=\Sigma_{\leq 2p-1-k}(S)$. If $S$ has the form
$$S=e_1^{[p-1]}\boldsymbol{\cdot} {\prod}^{\bullet}_{i\in[1,\ell]}{(a_ie_1+e_2})\boldsymbol{\cdot} {\prod}^{\bullet}_{i\in[1,v]}{(b_ie_1+x_ie_2}),$$
where $(e_1, e_2)$ is a basis of $G$, \ $\ell \geq p$, \ $a_i,b_i\in [1,p]$, \ $x_i\in [2,p-1]$, and $\sum_{i=1}^p{a_i}\equiv 1 \mod{p}$,  then $\textup h\left({\prod}^{\bullet}_{i\in[1,\ell]}{(a_ie_1+e_2)} \right)=p-1$.
\end{lemma}

\begin{proof}

If $v=0$, then we can apply Lemma \ref{car1} to complete the proof,  so we will assume $v\geq 1$.
Let $G'=C_p\oplus C_p \oplus C_p$ and let $(e_1, e_2, e_3)$ be a basis of $G'$. Let $\phi : G  \rightarrow G'$ be the map defined by $xe_1 + ye_2 \mapsto xe_1 + ye_2+e_3$ and let
\begin{equation}\label{3d-form1}
S' =\phi(S)\boldsymbol{\cdot} (-e_3)^{[p-k-1]} \boldsymbol{\cdot} ( -\sigma(S)-(2k-1)e_3) =S'_1 \boldsymbol{\cdot} S'_2 \boldsymbol{\cdot} S'_3 \boldsymbol{\cdot} S'_4 \boldsymbol{\cdot} S'_5,
\end{equation}
where
\begin{align}\label{3d-part1}
&S'_1 = \phi(e_1^{[p-1]})=(e_1+e_3)^{[p-1]},
\\\label{3d-part2}
&S'_2 = \phi\left({\prod}^{\bullet}_{i\in[1,\ell]}{(a_ie_1+e_2})\right)=
{\prod}^{\bullet}_{i\in[1,\ell]}{(a_ie_1+e_2+e_3}),
\\\label{3d-part3}
&S'_3 = \phi\left({\prod}^{\bullet}_{i\in[1,v]}{(b_ie_1+x_ie_2})\right)=
{\prod}^{\bullet}_{i\in[1,v]}{(b_ie_1+x_ie_2+e_3)},
\\
&S'_4 = (-e_3)^{[p-k-1]}, \quad \und
\\\label{3d-part5}
&S'_5 = -\sigma(S)-(2k-1)e_3=-\left(\sum_{i=1}^{\ell}{a_i} + \sum_{i=1}^v{b_i} -1\right)e_1 - \left(\ell+\sum_{i=1}^v{x_i}\right)e_2 -(2k-1)e_3.
\end{align}

\textbf{Claim A:}  $S'$ is a minimal zero-sum sequence of length $3p-2$.

\begin{proof}[Proof of Claim A]
 Since $\ell +v+p-1=|S|=2p-2+k$, then $|S'|=3p-2$. Also,
\begin{equation*}
\begin{split}
\sigma(S') & =(\sigma(S)+(2p-2+k)e_3)-(p-k-1)e_3-\sigma(S)-(2k-1)e_3  = pe_3  = 0,
\end{split}
\end{equation*} so $S'$ is zero-sum.
Furthermore, by Lemma \ref{mzs-exists}(c), $-\sigma(S)-(2k-1)e_3\neq 0$. Assume by contradiction that $S'$ has a proper, nonempty zero-sum subsequence $T'$. We will examine two cases.

\medskip

\noindent\textbf{Case 1:} Suppose $-\sigma(S)-(2k-1)e_3 \not\in \textup{supp}(T')$.

Then $T' = \phi(T)\boldsymbol{\cdot} (-e_3)^{[i]}$ where $i\in [0,p-k-1]$ and $T$ is a subsequence of $S$. Observe that $$0 = \sigma(T') = \sigma (T) + (|T|-i)e_3,$$
so $|T|\equiv i \mod{p}$, and $T$ is a nonempty zero-sum subsequence of $S$. From Lemma \ref{mzs-exists} part (a), $$i\equiv |T| \in [2p-k,2p-1]\cup [4p-2k,2p-2+k] \equiv [p-k,p-1] \mod{p},$$
with the latter congruence holding since $p-k\leq 2p-2k$ and $k-2\leq p-3$,
which is contrary to the definition of $i$.

\medskip

\noindent\textbf{Case 2:} Suppose $-\sigma(S)-(2k-1)e_3 \in \textup{supp}(T')$.

Then $T' = \phi(T)\boldsymbol{\cdot} (-e_3)^{[i]} \boldsymbol{\cdot} (-\sigma(S)-(2k-1)e_3)$ where $i\in [0,p-k-1]$ and $T$ is a subsequence of $S$. Observe that $$0 = \sigma(T') = \sigma (T)-\sigma(S) + (|T|-i-2k+1)e_3,$$ so $\sigma(T) = \sigma(S)$ and $|T| \equiv i +2k -1 \mod{p}$. Consider $T^{[-1]}\boldsymbol{\cdot} S$, which will be zero-sum. Also,
$$|T^{[-1]}\boldsymbol{\cdot} S|=2p-2+k-|T| \equiv 2p -k -1 - i \mod{p}.$$
If $T=S$, then $2p-2+k=|S|=|T|\equiv i+2k-1$ forces $i=p-k-1$, in which case $T'=S'$, contradicting that $T'$ is a proper zero-sum subsequence of $S'$. Therefore  $T^{[-1]}\boldsymbol{\cdot} S$ is a nonempty zero-sum subsequence of $S$, so
Lemma \ref{mzs-exists} parts (a) and (c) implies
 $$2p -k -1 - i \in [2p-k,2p-1]\cup [4p-2k,2p-2+k] \equiv [p-k,p-1] \mod{p}.$$
From this, we have that $ i \in [p-k,p-1] \mod{p}$, which is also contrary to the definition of $i$.
\end{proof}

By Claim A, $S'$ satisfies the hypothesis of Lemma \ref{mzs_3d}. Thus, by setting $$f_1=e_1 +e_3,$$ there are $f_2$ and $f_3$ with $(f_1,f_2,f_3)$ a basis for $G'$ such that
\be\label{grump}S'=f_1^{[p-1]}\boldsymbol{\cdot} {\prod}^{\bullet}_{i\in[1,p-1]}{(\alpha_if_1+f_2)}\boldsymbol{\cdot} {\prod}^{\bullet}_{i\in[1,p]}{(\beta_if_1+\gamma_if_2+f_3)},\ee
where $\alpha_i,\beta_i,\gamma_i\in [0,p-1]$.
Since $(e_1,e_2,e_3)$ is a basis for $G'$ with $f_1=e_1+e_3$, it follows that $(f_1,e_2,e_3)$ is also a basis for $G'$. Moreover, we can replace $f_2$ by $af_1+f_2$, for any $a\in [0,p-1]$, and \eqref{grump} remains true using this  alternative value of $f_2$, adjusting the coefficients $\alpha_i$ and $\beta_i$ appropriately. Thus, by choosing $a\in [0,p-1]$ appropriately, we can w.l.o.g. assume \be\label{how}f_2\in \la e_2\ra\oplus\la e_3\ra.\ee

Our goal now will be to determine $f_2$. By using the substitution $f_1=e_1+e_3$ in \eqref{3d-part1}--\eqref{3d-part5}, we obtain
\begin{align}\label{3d-part1-u1}
&S'_1 = f_1^{[p-1]},
\\
\label{3d-part2-u1}
&S'_2 = {\prod}^{\bullet}_{i\in[1,\ell]}{(a_if_1+e_2+(1-a_i)e_3}),\\
\label{3d-part3-u1}
&S'_3 = {\prod}^{\bullet}_{i\in[1,v]}{(b_if_1+x_ie_2+(1-b_i)e_3)},
\\
\label{3d-part4-u1}
&S'_4 = (-e_3)^{[p-k-1]}, \quad\und\\
\label{3d-part5-u1}
&S'_5  =-\left(\sum_{i=1}^{\ell}{a_i} + \sum_{i=1}^v{b_i} -1\right)f_1 - \left(\ell+\sum_{i=1}^v{x_i}\right)e_2 + \left( \sum_{i=1}^{\ell}{a_i} + \sum_{i=1}^v{b_i} - 2k\right)e_3.
\end{align}

Let $\pi : G' \rightarrow \left< e_2 \right> \oplus \left< e_3 \right>$ be the projection map defined by $xf_1 +ye_2 +ze_3 \mapsto ye_2 +ze_3$. Let
$$\Omega :=\pi(S'_2\boldsymbol{\cdot} S'_3 \boldsymbol{\cdot} S'_4 \boldsymbol{\cdot} S'_5) = \pi((f_1^{[p-1]})^{[-1]}\bdot S').$$ By \eqref{grump} and \eqref{how}, we have $\vp_{f_2}(\Omega)\geq p-1$.
Since $x_i \in [2,p-1]$, the supports of   $\pi(S'_2)$, $\pi(S'_3)$ and $\pi(S'_4)$ are pairwise disjoint with   $|\pi(S'_3)|=v=p+k-1-\ell < p-2$ and $|\pi(S'_4)|=p-k-1<p-2$, so the only way that $\vp_{f_2}(\Omega)\geq p-1$ is possible if either $\vp_{f_2}(\pi(S'_2))\geq p-1$, or else $\vp_{ f_2}(\pi(S'_2))=p-2$ and $\pi(S'_5)= f_2$.

If $\vp_{ f_2}(\pi'(S_2))\geq p-1$, then
$$p-1\leq \vp_{f_2}(\pi(S'_2)) \leq\textup h(\pi(S'_2))=\textup h\left({\prod}^{\bullet}_{i\in[1,\ell]}{(a_ie_1+e_2)} \right) \leq p-1,$$
where the equality in the middle is due to $a_ie_1+e_2$ in $S$ corresponding to $e_2 + (1-a_i)e_3$ in $\pi(S'_2)$, and the desired conclusion holds. Therefore we now assume
\be\label{what}\vp_{ f_2}(\pi(S'_2))=p-2\quad\und\quad \pi(S'_5)=f_2.\ee

Since $k\in [2,p-2]$ ensures that $p\geq 5$, we conclude from \eqref{what} that $ f_2$ is a term of $\pi(S'_2)$, whence  $f_2 = e_2 + (1-a_j)e_3$ for some $j\in [1,\ell]$. Now the term $a_ie_1+e_2$ in $S$ corresponds to the term $e_2+(1-a_i)e_3$ in $\pi(S'_2)$. We can replace the basis $(e_1,e_2)$ with the basis $(e_1,e'_2)$, where $e'_2=a_je_1+e_2$, and the hypotheses of the lemma remain valid replacing $a_i$ by $a'_i=a_i-a_j$ for $i\in [1,\ell]$, and likewise adjusting the values of the $b_i$. This leaves the value $f_1=e_1+e_3$ unchanged, with $f_2=e'_2-a_je_1+(1-a_j)e_3=e'_2+e_3-a_jf_1$. Thus, by also replacing $f_2$ by $f'_2=f_2+a_jf_1=e'_2+e_3$, and defining $\pi$ using the basis $(f_1,e'_2,e_3)$ rather than $(f_1,e_2,e_3)$, we can w.l.o.g. assume
that $$ f_2 = e_2 +e_3$$ with $a_i = 0$ for exactly $p-2$ values of $i\in [1,\ell]$, say w.l.o.g. $a_i=0$ for $i\in [1,p-2]$. Then we can rewrite \eqref{3d-part1-u1}--\eqref{3d-part5-u1} as follows:
\begin{align}\label{3d-part1-u2}
&S'_1 = f_1^{[p-1]},\\ \label{3d-part2-u2}
&S'_2 = f_2^{[p-2]}\boldsymbol{\cdot}{\prod}^{\bullet}_{i\in[p-1,\ell]}{(a_if_1+f_2-a_ie_3}),
\\
\label{3d-part3-u2}
&S'_3 = {\prod}^{\bullet}_{i\in[1,v]}{(b_if_1+x_if_2+(1-b_i-x_i)e_3)},
\\
\label{3d-part4-u2}
&S'_4 = (-e_3)^{[p-k-1]},\quad\und\\
\label{3d-part5-u2}
&S'_5  =-\left(\sum_{i=p-1}^{\ell}{a_i} + \sum_{i=1}^v{b_i} -1\right)f_1 - \left(\ell+\sum_{i=1}^v{x_i}\right)f_2 \\ &+ \left(\ell+\Sum{i=1}{v}x_i+ \sum_{i=p-1}^{\ell}{a_i} +\sum_{i=1}^v{b_i} - 2k\right)e_3
 =-\left(\sum_{i=p-1}^{\ell}{a_i} + \sum_{i=1}^v{b_i} -1\right)f_1 + f_2,\nn
\end{align}
where $a_i\in [1,p-1]$ for all $i\in[p-1,\ell]$, with the final equality in \eqref{3d-part5-u2} since $\pi(S'_5)=f_2$.

Since $f_1=e_1+e_3$ and $f_2=e_2+e_3$ with $(e_1,e_2,e_3)$ a basis for $G'$, it follows that $(f_1,f_2,e_3)$ is a basis for $G'$. Thus $f_3=a f_1+bf_2+c e_3$ for some $a,b\in [0,p-1]$ and $c\in [1,p-1]$, with $c\neq 0$ since $(f_1,f_2,f_3)$ is also a basis for $G'$. Letting $c^{-1}\in[1,p-1]$ be the multiplicative inverse of $c$ modulo $p$, we have $$e_3=(-c^{-1}a )f_1+(-c^{-1}b )f_2+(c^{-1})f_3.$$
In view of \eqref{grump} and \eqref{what}, all terms of $S'_3\bdot S'_4$ must have their $f_3$-coefficient, when written using the basis $(f_1,f_2,f_3)$, equal to  $1$. Likewise, all $\ell-(p-2)\geq 2$ terms $x$ of $S'_2$  with $\pi(x)\neq f_2$ must also have their $f_3$-coefficient, when written using the basis $(f_1,f_2,f_3)$, equal to $1$. As a result, substituting the value
$e_3=(-c^{-1}a )f_1+(-c^{-1}b )f_2+(c^{-1})f_3$ into \eqref{3d-part2-u2} yields $-a_ic^{-1}\equiv 1\mod p$ for all $i\in [p-1,\ell]$, while substituting into \eqref{3d-part4-u2} yields (in view of $k\leq p-2$) that $-c^{-1}\equiv 1\mod p$. It follows that $$c=-1\quad\und\quad a_i=1\quad\mbox{ for all $i\in [p-1,\ell]$}.$$
Recalling that $a_i=0$ for $i\in [1,p-2]$, we conclude that  $S$ has the form
\be\label{one}S=e_1^{[p-1]}\boldsymbol{\cdot} e_2^{[p-2]} \boldsymbol{\cdot} (e_1 + e_2)^{[\ell-p+2]}\boldsymbol{\cdot} {\prod}^{\bullet}_{i\in[1,v]}{(b_ie_1+x_ie_2}),\ee
where $\ell\geq p$ and $x_i\in [2,p-1]$ for all $i\in [1,v]$.

By Lemma \ref{mzs-exists} part (b) and $k \not\equiv 0 \mod{p}$, $S$ has a minimal zero-sum subsequence of length $2p-1$, say $T$. Note that $|S\boldsymbol{\cdot}\big(e_1^{[p-1]}\boldsymbol{\cdot}e_2^{[p-2]}\big)^{[-1]}|=k+1\leq p-1$. Thus,
in view of \eqref{one}, $\ell\geq p$ and $v\geq 1$, we see that
$e_1$ is the only term of $S$ with multiplicity $p-1$, while there are at most $v=|S|-(p-1)-\ell\leq |S|-2p+1= k-1\leq p-2$ terms of $S$ neither equal to $e_1$ nor from the coset $\la e_1\ra +e_2$. As a result,
   Theorem \ref{probB} implies that this zero-sum subsequence $T$ must have the form
$$T=e_1^{[p-1]} \boldsymbol{\cdot}  e_2^{[\alpha]} \boldsymbol{\cdot} (e_1 + e_2)^{[\beta]},$$
where $\alpha \in [0,p-2]$ and $\alpha+\beta =p$. But then $0=\sigma(T)=(\beta-1)e_1$, which implies that $\beta=1$ and $\alpha=p-1$, contradicting that $\vp_{e_2}(S)=p-2$ (in view of \eqref{one}), which completes the proof.
\end{proof}

\medskip

\ We can now prove our main result.

\medskip
\begin{proof}[Proof of Theorem \ref{main-result}] Since  $k\not\equiv 0 \mod{p}$, Lemma \ref{mzs-exists}(b) implies that $S$ contains a  minimal zero-sum subsequence of length $\textup D(G)=2p-1$, say $T$.  By Theorem \ref{probB}, there is a basis $(e_1,e_2)$ for $G$ such that $T=e_1^{[p-1]}\bdot {\prod}^{\bullet}_{i\in[1,p]}(a_ie_1+e_2)$, for some $a_i\in [1,p]$ with $\Sum{i=1}{p}a_i\equiv 1\mod p$, ensuring that $S$ satisfies the hypotheses of Lemma \ref{height-of-S}.
Note, there can be at most $p-1=\textup D(C_p)-1$ terms from $\la e_1\ra$ in $S$, else  $S$ contain a nonempty  zero-sum subsequence with length at most $p\leq 2p-1-k$, contrary to hypothesis. Lemma \ref{height-of-S} now implies that there is some term $e'_2:=ae_1+e_2$, where $a\in [1,p]$, having  multiplicity $p-1$ in $S$. Since $(e_1,e_2)$ is a basis for $G$, so too is $(e_1,e'_2)$, with $S=e_1^{[p-1]}\bdot {e'_2}^{[p-1]}\bdot T'$ for some subsequence $T'$ of $S$, allowing us to apply Lemma \ref{p-1_repeating} to yield the desired structure for $S$.
%
%
%
%
\end{proof}

\end{document}